\documentclass[12pt,reqno]{amsart}
\usepackage[usenames]{color}
\usepackage{amsmath, amsfonts, epsfig, amssymb, graphics}
\usepackage[colorlinks=true,citecolor=cyan]{hyperref}
\usepackage[mathscr]{euscript}

\def\mbf#1{{\mathbf{#1}}}

\newcommand{\CW}[1][W]{\mathscr{C}_{#1}}
\newcommand{\Harm}[1][W]{\mathscr{H}_{#1}}

\newcommand{\poly}{\mathcal{R}_n^{(r)}}
\newcommand{\Ideal}[1][W]{\mathscr{I}_{#1}}

\def\N{{\mathbb N}}

\def\C{{\mathbb C}}
\def\S{{\mathbb S}}

\def\scalar#1#2{{\langle#1,#2 \rangle}}

\newcommand{\qbinom}[2]{\genfrac{[}{]}{0pt}{}{#1}{#2}}
\newcommand{\charac}{\raise 2pt\hbox{$\chi$}}

\def\auteur#1{{\sc #1}}
\def\titreref#1{{\em #1}}
\def\vol#1{{\bf #1}}

\newcommand{\pref}[1]{{\rm (\ref{#1})}}
\def\defn#1{\bleu{\bf #1}}

\def\bleu{\textcolor{blue}}
\def\rouge{\textcolor{red}}

\newdimen\carrelength
\newtheorem{theorem}{\bleu{Theorem}}[section]

\newtheorem{lemma}{\bleu{Lemma}}[section]

\begin{document}

\title[\rouge{Multivariate Diagonal Harmonics}]{{\bleu{Multivariate Diagonal Coinvariant Spaces\\ for\\
Complex Reflection Groups}}}
\author[F.~Bergeron]{Fran\c{c}ois Bergeron}
\date{March 2009. This work is supported by NSERC-Canada}
\maketitle
\begin{abstract}
For finite complex reflection groups, we consider the graded $W$-modules of diagonally harmonic polynomials in $r$ sets of variables, and show that 
associated Hilbert series may be described in a global manner, independent of the value of $r$.
\end{abstract}
 \parskip=0pt

%%%%%%%%%%%%%%%%%%%%%%%%%%   
{ \setcounter{tocdepth}{1}\parskip=0pt\footnotesize \tableofcontents}
\parskip=8pt

%%%%%%%%%%%%%%%%%%%%%%%%%%%%%%%%%%%%%%%%%%%%
%%%%%%%%%%%%%%%%%%%%%%%%%%%%%%%%%%%%%%%%%%%%

\section{Introduction}
For  finite complex reflection groups $W=G(m,p,n)$, we study the diagonal coinvariant space $\CW$ for $W$, in several (say $r$)  sets  of $n$ variables. Here, the use of the term diagonal refers to the fact that $W$ is considered as a diagonal subgroup of $W^r$, acting on the $r^{\rm th}$-tensor power $\poly$ of the symmetric algebra of the defining representation of $W$.  
The space considered, $\CW^{(r)}$, is simply the quotient of $\poly$ by the ideal generated by constant-term-free (diagonal) $W$-invariants.
We shall see that the associated multigraded Hilbert series, denoted $\CW^{(r)}(q_1,\ldots,q_r)$ (which is symmetric in the variables  $\mbf{q}:=(q_1,\ldots,q_r$),  can be described in an uniform manner as a positive coefficient linear combination of Schur polynomials
    \begin{equation}\label{schur_generic}
              \bleu{\CW^{(r)}(\mbf{q}) = \sum_\mu c_\mu\, s_\mu(\mbf{q}),}
       \end{equation}
 with the $c_\mu$ independent of $r$, and $\mu$ running through a finite set of integer partitions that depend only on the group $W$. This expression has the striking feature that it gives one global formula that specializes to the dimension of $\CW^{(r)}$, for each individual $r$. 
 To better see what is striking here,  it is worth recalling that , although the case $r=1$ has a long history \cite{chevalley, sheppard, steinberg}, it is only recently that  the special case $r=2$ has been somewhat settled \cite{cherednik, gordon, griffeth, haimanvanishing}. However much still needs to be done along these lines as discussed in~\cite{HHLRU}. Some headway has recently been made in the case $r=3$ (see~\cite{trivariate, loktev}), but the general case is still wide open.

%%%%%%%%%%%%%%%%%%%%%%%%%%%%%%%%%%%
\section{Definitions and discussion}

For a rank $n$ complex reflection group $W$, we may consider its ``diagonal'' action on the $\N^r$-graded space $\poly:=\C[X]$,  of polynomials in the $r\times n$ variables 
            $$X=\begin{pmatrix}
               x_{11} & x_{12} & \cdots & x_{1n}\\
               \bleu{x_{21}}& \bleu{x_{22}} & \bleu{ \cdots} & \bleu{x_{2n}}\\
               \vdots & \vdots & \ddots & \vdots\\
               x_{r1} & x_{r2} & \cdots & x_{rn}
             \end{pmatrix}.$$
We consider each row of $X$ as a  set of $n$ variables. Thinking of elements $w\in W$ as $n\times n$  matrices, the action of $w$  is simply the map sending $f(X)$ to $f(X\,w)$. Naturally, \defn{$W$-invariant polynomials} in $\C[X]$ are those that fixed by any element of $W$, i.e.:
  \begin{displaymath} 
               f(X\cdot w)=f(X),\qquad \hbox{for all}\ w\in W.
  \end{displaymath} 
The \defn{diagonal coinvariant space} $\CW^{(r)}$ is defined to be the quotient
      \begin{equation}
          \bleu{\CW^{(r)}:=\poly/\Ideal^{(r)}},
       \end{equation}
where $\Ideal^{(r)}$ is the ideal generated by constant-term-free $W$-invariant polynomials in $\poly$. 

For an integer $r\times n$ matrix $A=(a_{ij})$, we denote by $X^A$ the monomial
   \begin{equation}
        \bleu{X^A:=X_1^{A_1} X_2^{A_2}\cdots X_n^{A_n}},
     \end{equation}
 where $X_j$ and $A_j$ respectively stand for the $j^{\rm th}$ column of $X$ and $A$, and where
   \begin{displaymath} \bleu{X_j^{A_j}:=\prod_{i=1}^r x_{ij}^{a_{ij}}}. \end{displaymath}
The ideal $\Ideal$ is homogeneous with respect to the (vector) \defn{degree} 
  \begin{equation}\label{defn_deg}
       \bleu{\deg(X^A):=\sum_{j=1}^n A_j},
  \end{equation}
and the \defn{total degree}, \bleu{${\rm tdeg}(X^A)$}, is the sum of the components of $\deg(X^A)\in N^r$.   The diagonal $W$-action is clearly degree preserving. 
 
 The space $ \CW^{(r)}$ may be turned into a polynomial representation of $GL_r$, using the fact that the usual action of $GL_r$ on $\poly$ (that sends $f(X)$ to $f(MX)$, when  $M\in GL_r$) commutes with the $W$-action. Hence the $GL_r$-action preserves the ideal $\Ideal^{(r)}$.  Recall that the \defn{$GL_r$-character} of  $ \CW^{(r)}$ is the symmetric polynomial $ \CW^{(r)}(\mbf{q})$, in the variables $\mbf{q}$,
obtained by taking the trace of the linear transform 
                 \begin{displaymath} f(X)\longmapsto f(QX),\end{displaymath} 
where $Q$ is the diagonal matrix $[q_1,\ \ldots\ ,q_r]$. 
Since characters of irreducible polynomial $GL_r$-modules correspond precisely to Schur polynomials $s_\mu(\mbf{q})$, we get an expansion of form \pref{schur_generic} for $ \CW^{(r)}(\mbf{q})$, with the $c_\mu$ giving the multiplicity of the representation having character $s_\mu(\mbf{q})$. 
Moreover, the ideal $\Ideal$ being homogeneous, the variable $q_i$ serves  as  a ``degree counter'' for the variables $x_{ij}$,  $1\leq j\leq n$, so that
we may also consider $\CW^{(r)}(\mbf{q})$ as the \defn{Hilbert series} of $\CW^{(r)}$. 
 
\subsection*{Harmonic polynomials} Some properties of $\CW^{(r)}$ are better formulated in the context of the isomorphic space $\Harm^{(r)}$, of ``diagonally harmonic polynomials''. 
In fact, we think of this new space as canonical representatives for elements of $\CW^{(r)}$.
More precisely, for each of the variables $x_{ij} \in X$, consider the
partial derivation denoted by $\partial{x_{ij}}$. Given polynomial $f(X)$, we denote by $f(\partial X)$  the
differential operator obtained by replacing the variables in $X$ by the
corresponding derivation.  The space $\bleu{\Harm^{(r)}}$ of
\defn{diagonally harmonic polynomials}  is the set of 
polynomial solutions, $g(X)$, of the system of partial differential
equations
\begin{equation}\label{defn_harmonic}
    \bleu{f(\partial X ) (g(X)) = 0},
\end{equation}
with one equation for each $ f(X)\in \Ideal$.
Evidently, we need only consider a generating set of $\Ideal$ for these equations to characterize all solutions. 
An elementary proof (see~\cite{livre}) that $\CW^{(r)}$ and $\Harm^{(r)}$ are isomorphic relies on the fact that $\Harm^{(r)}$ appears as the orthogonal complement of $\Ideal$ for the scalar product
\begin{equation}\label{def_scalar2}
    \bleu{ \scalar{f(X)}{g(X)}=f(\partial X)g(X)\big|_{X=0}},
 \end{equation}
where $\bleu{p(X)\big|_{X=0}:=p(0)}$ is the \defn{constant term} of $p(X)$. 
It is easy to see that the set of monomials forms an orthogonal basis for this scalar product, and hence deduce that it is $W$-invariant. One also checks readily that 
\begin{lemma}
$\Harm^{(r)}$ is the orthogonal complement of $\Ideal^{(r)}$:
               \begin{equation}\bleu{ \Harm^{(r)}=\left( \Ideal^{(r)} \right)^\perp}.\end{equation}  
\end{lemma}
\begin{proof}[\bleu{\bf Proof}.]
 Indeed,  let $g(X)$ be in $\Ideal^\perp$, and consider the leading term (for any suitable term order) of
$f(\partial X)g(X) = c\, X^A+ \ldots$,
 if any, for some $f(X)$ in $\Ideal$.
Since $X^A\,f(X)$ also lies in $\Ideal$, we must have $\scalar{X^A\,f_i(X)}{g(X)} =0$, but this means precisely that $c=0$. Hence we must have $f(\partial X)G(X)=0$. 
\end{proof}

The case $r=1$, i.e.: $\Harm^{(1)}$,  gives rise to the classical theorems of \cite{sheppard, weyl} regarding $W$-harmonic polynomials and the coinvariant space for $W$. The case $r=2$, for the symmetric group $W=\S_{n}$, corresponds to  the space of diagonal harmonic polynomials of Haiman-Garsia, which is of dimension $(n+1)^{n-1}$. Its alternating part has bigraded Hilbert series given by the now famous $q,t$-Catalan polynomials $C_{n+1}(q,t)$.  Again for $r=2$, the case of other reflection groups has also been studied for other groups $W$ (see~\cite{gordon} for instance).

Any $f(X)$ lying in $\Harm^{(r-1)}$ also lies in $\Harm^{(r)}$, so that we have
              \begin{displaymath}\bleu{ \Harm^{(1)}\subseteq \Harm^{(2)}\subseteq \ldots \subseteq  \Harm^{(r)}\subseteq \ldots }.\end{displaymath}
In particular, 
\begin{equation}\label{restriction}
   \bleu{\Harm^{(r-1)}(q_1,\ldots,q_{r-1})=\Harm^{(r)}(q_1, \ldots,q_{r-1},0)}.
 \end{equation}
It is easy to adapt an argument of~\cite{garsia_haiman}, for the case $r=2$, to show that the total degree of elements of $\Harm^{(r)}$ is bounded by the maximal total degree of an element of $\Harm^{(1)}$ (which is well known to be the degree of the Jacobian of $W$). Hence, only a finite number of Schur functions may appear in  the expansions of the $\Harm^{(r)}(\mbf{q})$, $r\geq 1$. This implies that $\Harm^{(r)}(\mbf{q})$ affords an expansion in the $s_\mu(\mbf{q})$, with $\mu$ independent of $r$. For this reason, we say that  this expansion is \defn{universal}, and we drop the ``$(r)$'' superscript and write simply $\Harm(\mbf{q})$. In fact, it follows from these considerations, and further discussion in section~\ref{preuves} that we have the following.

\begin{theorem}\label{mainthmhilb}
   For any rank $n$ complex reflection group $W=G(p,1,n)$,  the Hilbert series of the diagonal coinvariant space in $r$ sets of variables affords an expansion, in terms of Schur functions $s_\mu(q_1,\ldots ,q_r)$,  with  positive integer coefficients $c_\mu$ that are independent of $r$, the sum being over the set of partitions $\mu$ of integers $d$:
     \begin{equation}\label{interval_entiers}
      \bleu{ 0\leq d \leq \frac{n(r\,n+r-2)}{2}},
     \end{equation}
   and having at most $\bleu{n}$ parts.
\end{theorem}

In many of the cases considered here,  the graded Hilbert series of $\Harm$ seems to  take the form
  \begin{equation}\label{hpositif}
         \bleu{\Harm(\mbf{q})= \sum_{\mu} a_\mu\, h_\mu(\mbf{q})},
  \end{equation}
with the sum being over a finite set of partitions, and the $a_\mu$ positive integers. When this last property holds, one says that $ \Harm(\mbf{q})$ is \defn{$h$-positive}.
In the case of the symmetric group (at least), writing $\Harm[n]$ for $\Harm[\S_n]$, an even finer $h$-positivity phenomenon seems to occur. It involves the decomposition into irreducibles of the homogeneous components of $\Harm[n]$. This is all encompassed into the \defn{graded Frobenius characteristic} of $\Harm[n]$:
    \begin{displaymath}\bleu{\Harm[n](\mbf{w};\mbf{q}):=\sum_{d\in \N^3} \mbf{q}^\mbf{d} \mathcal{F}_{\Harm[n;\mbf{d}]}(\mbf{w})}.\end{displaymath}
Recall that the Frobenius characteristic $\mathcal{F}_\mathcal{V}(\mbf{w})$, of a $\S_n$-module $\mathcal{V}$, is the symmetric function (in auxiliary variables $\mbf{w}=w_1,w_2,\ldots$) whose expansion in terms of the Schur functions $S_\lambda(\mbf{w})$ records the multiplicity of irreducibles in $\mathcal{V}$. This is to say that we have
     \begin{displaymath} \mathcal{F}_\mathcal{V}(\mbf{w})= \sum_{\lambda\vdash n} b_\lambda\, S_\lambda(\mbf{w}),\end{displaymath}
  with the sum being over partitions of $n$, and $b_\lambda$ giving the multiplicity of the irreducible representation classified by $S_\lambda(\mbf{w})$. Our previous arguments show that there is a universal expansion of $\Harm[n](\mbf{w};\mbf{q})$ in terms of the $S_\lambda(\mbf{w})$, having Schur positive coefficients in the $s_\mu(\mbf{q})$. This is to say that
  \begin{theorem}\label{mainthmfrob}
   \begin{equation}
     \bleu{ \Harm[n](\mbf{w};\mbf{q})=\sum_{\lambda\vdash n} \Big(\sum_{\mu} b_{\lambda,\mu} s_\mu(\mbf{q})\Big) S_\lambda(\mbf{w})},\qquad n_{\lambda,\mu}\in\N.
   \end{equation}
 where the $n_{\lambda,\mu}$ are independent of $r$, with $\mu$ running over partitions of $d$ $(0\leq d\leq \binom{n}{2})$, having at most $n$ parts.
\end{theorem}
We underline that there are two kinds of Schur function at play here. Those in the $\mbf{q}$-variables (denoted by a lower case ``s'' and with the $\mbf{q}$-variables dropped),  that account  for graded multiplicities, and those in the $\mbf{w}$-variables,  that account for the decomposition into $\S_n$-irreducibles. 
For $n$ up to $5$, the expansion of $\Harm[n](\mbf{w};\mbf{q})$ in terms of the monomial symmetric functions $m_\lambda(\mbf{w})$ is $h$-positive in the $\mbf{q}$ variables. In formula,
   \begin{equation}\label{formule_conj}
      \bleu{ \Harm[n](\mbf{w};\mbf{q})=\sum_{\lambda\vdash n} \Big(\sum_{\mu} a_{\lambda,\mu} h_\mu(\mbf{q})\Big) m_\lambda(\mbf{w})},
   \end{equation}
 with $a_{\lambda,\mu}\in\N$ independent of $r$.
For example, we have
\begin{equation}\label{valeurs_Frob_mh}
    \Harm[3](\mbf{w};\mbf{q}) = m_{{3}}(\mbf{w})+ \left( 1+h_{{1}}+h_{{2}} \right) m_{{21}}(\mbf{w})+ \left( 1+2\,h_{{1}
}+h_{{2}}+h_{{3}}+{h_{{11}}} \right) m_{{111}}(\mbf{w}),
\end{equation}
once again with the $\mbf{q}$-variables dropped. Since the coefficient of $m_{11\cdots 1}(\mbf{w})$ is the Hilbert series of the underlying representation, the $h$-positivity of~\pref{formule_conj} implies that $\Harm[n](\mbf{q})$ is also $h$-positive, on top of being universal. Hence, in such cases the Hilbert series of $\Harm[n]$ would have to take the form 
  \begin{equation}\label{formule_hilbert}
       \bleu{  \Harm[n](\mbf{q})= \sum_{\sigma\in \S_n} h_{\mu(\sigma)}(\mbf{q})},
  \end{equation}
with $\mu(\sigma)$ some partition of the number of inversions of $\sigma$.
Direct calculations give the universal expansions
 \begin{equation}\label{cas_S_n}
\begin{array}{rcl}
    \Harm[1](\mbf{q}) &=&1,\\[4pt]
  \Harm[2](\mbf{q})&=&1+h_1,\\[4pt]
   \Harm[3](\mbf{q})&=&1+2\,h_1+h_2+h_{11}+h_3,\\[4pt]
   \Harm[4](\mbf{q})&=& 1+3\,h_1+2\,h_2+3\,h_{11}+2\,h_3+3\,h_{21}+h_{111},\\
    &&\qquad \qquad+h_4+4\,h_{31}+2\,h_5+h_{41}  +h_6 ,\\[4pt]
   \Harm[5](\mbf{q})&=&1+4\,h_{1}+3\,h_{2}+6\,h_{11}+3\,h_{3}+8\,h_{21}+
4\,h_{111}\\
\qquad\qquad&& +2\,h_{4}+9\,h_{31}+2\,h_{2,2}+6\,h_{211}+h_{1111}\\
\qquad\qquad&& +3\,h_{5}+4\,h_{41}+5\,h_{3,2}+10\,h_{311}\\
\qquad\qquad&& +h_{6}+ 9\,h_{51}+h_{4,2}+5\,h_{411}+4\,h_{3,3}\\
\qquad\qquad&& +2\,h_{7}+9\,h_{61}+2\,h_{52}+h_{511}+h_{4,3}\\
\qquad\qquad&& +4\,h_{8}+4\,h_{71}+h_{62}+3\,h_{9}+h_{81}+h_{\overline{10}},
 \end{array}
 \end{equation}
with $h_{\overline{10}}$ being indexed by a one part partition.
From the above data, we might expect (as was conjectured in first drafts of this paper) that we always have $h$-positivity, but  this fails to hold in general. Indeed\footnote{As recently calculated by M.~Haiman.}, the degree $9$ term of the Hilbert series $\Harm[6](\mbf{q})$ is
\begin{eqnarray*}
   &&\bleu{\rouge{-h_{{9}}}+18\,h_{{1}}h_{{8}}+2\,h_{{2}}h_{{7}}+17\,h_{{3}}h_{{6}}+7\,h_{{4}}h_{{5}}}\\
   &&\qquad\ \bleu{+28\,{h_{{1}}}^{2}h_{{7}}+12\,h_{{1}}h_{{2}}h_{{6}}+5\,h_{{1}}h_{{3}}h_{{5}}+h_{{1}}{h_{{4}}}^{2}+{h_{{1}}}^{3}h_{{6}}.
}
\end{eqnarray*}
Still, as discussed in section~\ref{low}, low degree terms of $\Harm[n](\mbf{w},\mbf{q})$ are $h$-positive. 

    Specializing the universal formula \pref{mainthmfrob}, we get the following polynomial formulas (in the parameter $r$) for the dimension of the spaces $\Harm^{(r)}$:
\begin{equation}\label{formule_dim}
 \begin{array}{rcl}
\dim \Harm[1] &=& 1\\[4pt]
\dim \Harm[2] &=& 1+r\\[4pt]
\dim \Harm[3] &=& (1+r)^2+\binom{r+1}{2}+\binom{r+2}{3}\\[4pt]
\dim \Harm[4] &=& (1+r)^3+2\binom{r+1}{2}+3\,r\binom{r+1}{2}+2\binom{r+2}{3}\\[4pt]
    &&\qquad +4\,r\binom{r+2}{3}+\binom{r+3}{4}+r\binom{r+3}{4}+2\binom{r+4}{5}+\binom{r+5}{6}
 \end{array} 
 \end{equation}
In particular, at $r=1$ these expressions evaluate to $n!$, and at $r=2$ they evaluate to $(n+1)^{n-1}$. In~\cite{trivariate}, we discuss the apparent fact that,
at $r=3$, formulas~\pref{formule_dim} should further specialize to $2^n(n+1)^{n-2}$. There is apparently no such nice formula for $r>3$. 

To better see how we may verify the above formulas, let us observe that we have the following basis for $ \Harm[2]$
   \begin{displaymath}
          \{1,x_{12}-x_{11},x_{22}-x_{21},\ldots, x_{r2}-x_{r1}\}.
    \end{displaymath}
Taking into account the action of $\S_2$, we can then easily calculate  that
   \begin{eqnarray}
        \Harm[2](\mbf{w};\mbf{q})&=& S_2(\mbf{w})+(q_1+q_2+\ldots+q_r)\,S_{11}(\mbf{w})\\
                                                 &=& m_2(\mbf{w})+ (1+h_1)\,m_{11}(\mbf{w}),
    \end{eqnarray}
 as well as
   \begin{displaymath}
          \Harm[2](\mbf{q})=1+q_1+q_2+\ldots+q_r.
    \end{displaymath}
Similar explicit calculations (with some help from the computer) give the universal expansions
\begin{equation}\label{multigraded}
  \begin{array}{rcl}
  \Harm[3](\mbf{w};\mbf{q})& =& S_3(\mbf{w}) + \bleu{\big(}{s_2+s_1}\bleu{\big)}\cdot S_{{21}}(\mbf{w})  
                + \bleu{\big(}{s_3+s_{11}} \bleu{\big)}\cdot S_{{111}}(\mbf{w}), \\[4pt]
  \Harm[4](\mbf{w};\mbf{q})& =&S_4(\mbf{w})+ \bleu{\big(}{s_3+s_2+s_1 }\bleu{\big)}\cdot S_{{31}}(\mbf{w})\\[4pt]
      &&\quad+  \bleu{\big(}{s_{4}+s_{21}+s_2} \bleu{\big)}\cdot  S_{{22}}(\mbf{w})\\[4pt]
      &&\quad + \bleu{\big(}{s_5+s_4+s_{31}+s_3+s_{21}+s_{11}}\bleu{\big)}\cdot S_{{211}}(\mbf{w})\\[4pt]
      &&\quad + \bleu{\big(}{s_{6}+s_{41}+s_{31}+s_{111} }\bleu{\big)}\cdot S_{{1111}}(\mbf{w}).\\[4pt]
  \Harm[5](\mbf{w};\mbf{q})&=& S_{{5}}(\mbf{w})+ \bleu{\big(} s_{{4}}+s_{{3}}+s_{{2}}+s_{{1}}\bleu{\big)}S_{{41}}(\mbf{w})\\[4pt]
             &&\quad+  \bleu{\big(} s_{{6}}+s_{{5}}+s_{{41}}+s_{{4}}+s_{{31}}+s_{{22}}+s_{{3}}+s_{{21}}+s_{{2}} \bleu{\big)} S_{{32}}(\mbf{w}) \\[4pt]   
             &&\quad+  \bleu{\big(} s_{{7}}+s_{{6}}+s_{{51}}+2\,s_{{5}}+s_{{32}}+s_{{41}}+s_{{4}}\\
                     &&\qquad\qquad\qquad\qquad\qquad\qquad  +2\,s_{{31}}+s_{{3}}+s_{{21}}+s_{{11}} \bleu{\big)} S_{{311}}(\mbf{w}) \\[4pt]   
            &&\quad+  \bleu{\big(} s_{{8}}+s_{{7}}+s_{{61}}+s_{{6}}+2\,s_{{51}}+s_{{42}}+s_{{5}}+2\,s_{{41}}\\
                     &&\qquad\qquad +s_{{32}}+s_{{311}}+s_{{4}}+s_{{31}}+s_{{22}}+s_{{211}}+s_{{21}} \bleu{\big)} S_{{221}}(\mbf{w}) \\[4pt]

             &&\quad+  \bleu{\big(} s_{{9}}+s_{{8}}+s_{{71}}+s_{{7}}+2\,s_{{61}}+s_{{52}}+s_{{6}}+2\,s_{{51}}+s_{{42}}\\
                     &&\qquad +s_{{411}}+s_{{33}}+2\,s_{{41}}+s_{{32}}+s_{{311}}+s_{{31}}+s_{{211}}+s_{{111}} \bleu{\big)} S_{{2111}}(\mbf{w}) \\[4pt]

             &&\quad+  \bleu{\big(} s_{{\overline{10}}}+s_{{81}}+s_{{71}}+s_{{62}}+s_{{61}}+s_{{511}}+s_{{43}}+s_{{42}}\\
                     &&\qquad\qquad\qquad\qquad\qquad\qquad\qquad  +s_{{411}}+s_{{311}}+s_{{1111}}\bleu{\big)}(\mbf{w}) S_{{11111}}
\end{array}
\end{equation}
where, once again, we write $s_{\overline{10}}$  to make clear that the index is a one part partition.
  We may also specialize to $1$ all  the $q_i$ in $\Harm[n](\mbf{w};\mbf{q})$, using the well known evaluation
\begin{equation}
   s_\mu(1^k)= s_\mu(\underbrace{11\cdots1}_{r\ {\rm copies}})=\prod_{(i,j)\in \mu} \frac{r+j-1}{h_{ij}(\mu)},
 \end{equation}
with $h_{ij}(\mu)$ denoting the hook length associated to the cell $(i,j)$ in the diagram of $\mu$. The resulting expressions have coefficients that are polynomials in $r$. For example,
\begin{equation}\label{multiplicities}
  \begin{array}{rcl}
  \Harm[2](\mbf{w};1^r)& =&S_2(\mbf{w}) +r\, S_{{11}}(\mbf{w}), \\[4pt]
  \Harm[3](\mbf{w};1^r)& =& S_3(\mbf{w}) +\frac{1}{2}r \left( r+3 \right) S_{21} (\mbf{w}) +\frac{1}{6}r \left( r^2+6\,r-1\right)S_{111} (\mbf{w}), \\[4pt]
  \Harm[4](\mbf{w};1^r)& =&S_4(\mbf{w})+ \frac{1}{6}r \left(r^2+6\,r+11 \right) S_{{31}}(\mbf{w})\\[4pt]
 &&\quad+  \frac{1}{24}r \left( r+1 \right)  \left( r^2+13\,r+10 \right)  S_{{22}}(\mbf{w}),\\[4pt]
&&\quad + \frac {1}{120}r \left( r+3\right)  \left( {r}^{3}+27\,{r}^{2}+74\,r-12 \right) S_{{211}}(\mbf{w}),\\[4pt]
&&\quad + \frac {1}{720}r \left( {r}^{5}+39\,{r}^{4}+295\,{r}^{3}+645\,{r}^{2}-296\,r+36 \right)S_{{1111}}(\mbf{w}),
\end{array}
\end{equation}
At $r=2$ the coefficient of $S_{11\cdots 1}$, in  $\Harm[n](\mbf{w};1^r)$, is the $n^{\rm th}$ Catalan number, and at $r=3$ it appears to be the number of intervals in the Tamari Lattice (see~\cite{trivariate}).
 
 We may also specialize formulas (\ref{multigraded}) by setting $q_1=t$, and $q_i=0$ for $i\geq 2$. From classical results on the coinvariant space for $\S_n$, as well as typical calculations on symmetric functions\footnote{For more on this, see   \cite{livre}.}, we get
  \begin{equation}\label{specun}
       \bleu{\Harm[n](\mbf{w};t)=\sum_{\lambda\vdash n} \qbinom{n}{\lambda}_t\,m_\lambda(\mbf{w})},
    \end{equation}
 where 
  \begin{displaymath}
     \bleu{\qbinom{n}{\lambda}_t:=\frac{n!_t}{\lambda_1!_t\cdots \lambda_k!_t}},
 \end{displaymath} 
with
  \begin{displaymath}
     \bleu{n!_t:=(1+t)(1+t+t^2)\cdots (1+t+\ldots+t^{n-1})}.
 \end{displaymath} 
In view of results in \cite{haimanhilb}, we must also have that
  \begin{equation}
       \Harm[n](\mbf{w};q,t)=\nabla(e_{n+1}(\mbf{w}))
    \end{equation}
 where $\nabla$ is an operator on symmetric functions, having Macdonald symmetric functions as eigenfunctions. This imposes further constraints on the form of Formula~\pref{mainthmfrob}.

Similar situations are settled by the following for two infinite families of groups.
\begin{theorem}\label{thm_I2m}
For the dihedral groups $I_2(m)=G(m,m,2)$, the cyclic groups $\mathcal{C}_m=G(m,1,1)$, and the groups $G(m,1,2)$, we have the respective universal $h$-positive Hilbert series
  \begin{eqnarray}
        \bleu{\Harm[\mathcal{C}_n](\mbf{q})}&=&\bleu{ \sum_{j=0}^{n} h_j(\mbf{q})},\label{formule_cyclic}\\
        \bleu{\Harm[I_2(m)](\mbf{q})}&=&\bleu{1 + 2\,h_1(\mbf{q})+h_{11}(\mbf{q})+h_2(\mbf{q})  +2\, \sum_{j=3}^{m-1} h_j(\mbf{q}) + h_m(\mbf{q})},
                      \label{formule_I2m}\\
        \bleu{\Harm[G(m,1,2)](\mbf{q})}&=&\bleu{ \Big( \sum _{k=0}^{m-1}h_{{k}} \Big) ^{\!\!2}+\sum _{k=0}^{m-1} \left( k+1 \right) h_{{m+k}}+\sum _{k=1}^{m-1} \left( m-k \right) h_{{2\,m-1+k}}}\label{formule_Gm12}
           \end{eqnarray}
 \end{theorem}

%%%%%%%%%%%%%%%%%%%%%%%%%%%%%%%%%%%%%%%%%%%%
\section{Proofs}\label{preuves}
\begin{proof}[\bleu{\bf Proof of Theorem~\ref{mainthmhilb} and ~\ref{mainthmfrob}}] In both theorems, it remains only to show that the $\mbf{q}$-variables Schur polynomials involved are indexed by partitions having at most $n$-parts, since every such Schur polynomial occurs as the character of an irreducible $GL_r$-sub-representation of the space $\poly$. Indeed, this property holds the Hilbert series (or $GL_r$-character) of the space $\poly$
which is given by the following formula
   \begin{equation}\label{hilbpoly}
           \bleu{\poly(\mbf{q})= h_n\left[ H(\mbf{q})\right]},
    \end{equation}
where 
    \begin{equation}
           \bleu{H(\mbf{q})=\sum_{k=0}^\infty h_k(\mbf{q}) = \prod_{i=1}^r \frac{1}{1-q_i}}.
    \end{equation}
 In Equation~\pref{hilbpoly}, we use a plethystic substitution notation. This means that, in order to calculate the right-hand side, we simply expand every symmetric functions in terms of power sum symmetric functions $p_k=p_k(\mbf{q})$ and apply the following rules:
 \begin{itemize}\itemsep=4pt
 \item $p_k[c]=c$ if $c$ is a constant;
  \item $p_k[p_j]=p_{k\cdot j}$;
  \item $F[f+g]=F[f]+F[g]$, and $(F+G)[f]=F[f]+G[f]$;
  \item $F[f\cdot g]=F[f]\cdot F[g]$, and $(F\cdot G)[f]=F[f]\cdot G[f]$.
    \end{itemize}
For the purpose of such calculations, $\mbf{q}$ is identified with the sum $q_1+\ldots +q_r=p_1$.
Using classical calculations on symmetric functions (see~\cite{macdonald}, we may then check that the only Schur polynomials that occur in the expansion of ~\pref{hilbpoly} are indexed by partitions having at most $n$ parts.\end{proof} 

\begin{proof}[\bleu{\bf{Proof of~\pref{formule_hilbert}}}] To prove Proposition~\ref{positivite_inv}, we observe that, for $r=1$, $\Harm(\mbf{q})$ specializes to the well known  Poincar\'e polynomial of $W$, in the variable $t=q_1$. Thus we must have
\begin{equation}\label{poincare_polynomial}
   \Harm(t):= \sum_{w\in W} t^{\ell(w)}= \prod_{i=1}^n\frac{t^{d_i}-1}{t-1},
\end{equation}
with the $d_i$'s standing for the \defn{degrees}  of the group $W$, and $\ell(w)$ is the length function. Perforce, if the universal Hilbert series $\Harm(\mbf{q})$ is $h$-positive, such as in \pref{hpositif}, then the positive integers $a_\mu$ must be such that
   \begin{displaymath}
       \sum_\mu a_\mu t^{|\mu|} = \sum_{w\in W} t^{\ell(w)}.
   \end{displaymath}
Indeed, the evaluation of $h_\mu$ in one variable $t$ is exactly $t^{|\mu|}$, where $|\mu|$ stands for the sum of the parts of $\mu$. Moreover, for the symmetric group, the ``length'' function $\ell(\sigma)$ corresponds to the number of inversion in $\sigma$. 
\end{proof}

\begin{proof}[\bf \bleu{Proof of Theorem~\ref{thm_I2m}}.]  
Each of the formula is obtained by constructing an explicit basis of the associated space.
The cyclic group $\mathcal{C}_m=G(m,1,1)$ case is almost immediate, since there is but one variable in each of the $r$-sets. The ring of diagonal $\mathcal{C}_m$-invariants is easily seen to be spanned by the set of monomials of total  degree $km$, in the $r$-variables $x_i=x_{i1}$, with $k\in N$. 
It follows that the associated diagonal harmonics are all monomials of degree at most $m-1$, hence the formula.

For both the Dihedral groups and $G(m,1,2)$, each set of variables contains two variables, which we denote by $x_i$ and $y_i$. Just as for the cyclic group case, the construction of an explicit basis of $\Harm$ is relatively easy. We illustrate for the dihedral case. By polarization (see~\cite{weyl}), we get that the ring of diagonal invariant of $I_2(m)$ is generated by the polynomials appearing as coefficients (themselves polynomials in the $x,y$-variables) of the polynomials in the $t_i$-variables:
        \begin{displaymath} 
              \bleu{\Big( \sum_{i=1}^r x_{i}\,t_i\Big)^m+ \Big( \sum_{i=1}^r y_{i}\,t_i\Big)^m},\qquad {\rm and}\qquad
              \bleu{\Big( \sum_{i=1}^r x_{i}\,t_i\Big)\cdot \Big( \sum_{i=1}^r y_{i}\,t_i\Big)}.
        \end{displaymath}
  It is easy to describe explicitly all the polynomial solutions of the resulting partial differential equations.
  
Now, expanded in term of Schur polynomials, Formula \pref{formule_I2m} takes the form
 \begin{equation}
     \bleu{\Harm[I_2(m)](\mbf{q})}=\bleu{1 + s_{11}+s_m+2\,\sum_{k=1}^{m-1} s_k}.
      \end{equation}
 We need only identified the highest weight vectors in $\Harm[I_2(m)]$ associated to each term of this last expansion. All of the terms having one index are easy to identify, they correspond to classical harmonic polynomials for $I_2(m)$. The only term of exception is $s_{11}$. It is readily seen to account for the irreducible component spanned by the polynomials in the set
     \begin{displaymath}\{x_iy_j-x_jy_i\ | \ 1\leq i<j \leq j \}.\end{displaymath} 
  One then checks directly that all diagonal harmonics are accounted in this manner.
\end{proof} 
%%%%%%%%%%%%%%%%%%%%%%%%%%%%%%%%%%%%%%%%%
%%%%%%%%%%%%%%%%%%%%%%%%%%%%%%%%%%%%%%%%%
\section{Low degree components}\label{low}
Even tough we restrict the discussion in this section to the case $W=\S_n$, much of it holds in generality. Our intent here is to compare the spaces $\mathcal{R}_n$
and $\Harm[n]\otimes \mathcal{R}_n^{\S_n}$, where $\mathcal{R}_n^{\S_n}$ is the ring of diagonal $\S_n$-invariants in $r$ sets of variables. In the case $r=1$, it is well known that these two spaces are isomorphic as graded $\S_n$-modules (see~\cite{chevalley}). In other words, $\mathcal{R}_n$ is a free module over the ring of symmetric polynomials.
This immediately implies that we have the explicit formula for 
\begin{equation}
       \bleu{\Harm[n](\mbf{w};q,0,0,\ldots) =  h_n\!\left[ \mbf{w}(1-q)^{-1}\right]\  \prod_{k=1}^n (1-q^k)}.
\end{equation} 
For $r\geq 2$, the space $\mathcal{R}_n$ is not a free module over the ring of diagonally symmetric polynomials. However, we do have
an isomorphism between the low degree homogeneous components of 
$\mathcal{R}_n$ and $\Harm[n]\otimes \mathcal{R}_n^{\S_n}$. This would be made more precise if one could calculate explicitly a free resolution of the quotient involved.
This is the subject of future work. For the time being, let us only observe that this leads to polynomial expressions in $n$ for the low degree coefficients of the Frobenius characteristic $\Harm[n](\mbf{w};\mbf{q})$ and the Hilbert series $\Harm[n](\mbf{q})=\Harm[n](\mbf{q})$.
These are derived using the following explicit expressions for $\mathcal{R}_n(\mbf{w};\mbf{q})$ and $\mathcal{R}_n^{\S_n}(\mbf{q})$:
\begin{equation}
    \bleu{\mathcal{R}_n(\mbf{w};\mbf{q})= h_n\![ \mbf{w}H(\mbf{q})]},\qquad {\rm and}\qquad 
    \bleu{\mathcal{R}_n^{\S_n}(\mbf{q})=   h_n\![ H(\mbf{q})]}.
\end{equation}
We thus get an approximation 
\begin{equation}\label{aprox_frob}
    \bleu{\Harm[n](\mbf{w};\mbf{q})\simeq_{n} \frac{h_n\![ \mbf{w}H(\mbf{q})]}{h_n\![ H(\mbf{q})]}},
\end{equation}
 that experimentally seems to hold for all terms of $\mbf{q}$-degree less or equal to $n$.
 Equivalently, using one of the Cauchy identities (expressing the fact that $\{h_\lambda\}_\lambda$ and $\{m_\lambda\}_\lambda$ are dual bases):
  \begin{equation}
    \bleu{ h_n\![ \mbf{x}\mbf{y}] = \sum_{\lambda\vdash n}  h_\lambda(\mbf{x}) m_\lambda(\mbf{y})},
\end{equation}
  the right-hand side of~\pref{aprox_frob} may be expanded as
 \begin{equation}
    \bleu{ \sum_{\lambda\vdash n} \frac{h_\lambda\![ H(\mbf{q})]}{h_n\![ H(\mbf{q})]}\ m_\lambda(\mbf{w})}.
\end{equation}
In particular, we may approximate the Hilbert series of $\Harm[n]$ by the expansion
  \begin{equation}\label{formule_hilb}
  \begin{array}{rcl}
      \bleu{\displaystyle\frac{ H(\mbf{q})^n}  {h_n\![ H(\mbf{q})]}}&=&\bleu{ 1+ ( n-1) h_1+ ( n-2 ) h_2+\binom{n-1}{2}\, h_1^{2}}\\[6pt]
&&\qquad   \bleu{+ ( n-1 ) ( n-3 )\, h_1h_2+ ( n-2 ) h_3+\binom{n-1}{3}\, h_1^3  +\ldots }
  \end{array}
\end{equation}
Observe here that the right-hand side of \pref{formule_hilb} gives explicit expressions, as polynomials in $n$, for the coefficient of $h_\mu$ in $\Harm[n](\mbf{q})$. These hold for $n\geq 3$. Observe that, using the other Cauchy identity
   \begin{equation}
    \bleu{ h_n\![ \mbf{x}\mbf{y}] = \sum_{\lambda\vdash n}  s_\lambda(\mbf{x}) s_\lambda(\mbf{y})},
\end{equation} 
 we may get a similar formula, describing the small degree isotopic $\S_n$-components of $\Harm[n]$, in the form 
\begin{equation}
    \bleu{\Harm[n](\mbf{w};\mbf{q})\simeq_{n}  \sum_{\lambda\vdash n} \frac{s_\lambda\![ H(\mbf{q})]}{h_n\![ H(\mbf{q})]}\ S_\lambda(\mbf{w})}.
\end{equation}
In this way, we may calculate the following order $4$ approximation
\begin{equation}
\begin{array}{rcl}
    \bleu{\Harm[4](\mbf{w};\mbf{q})}& \simeq_{4} & 
    \bleu{S_{{4}}(\mbf{w})+ \left( s_{{1}}+s_{{2}}+s_{{3}} \right) S_{{31}}(\mbf{w}) + }\\[4pt]
    &&\bleu{ \left( s_{{2}}+s_{{21}}+s_{{4}} \right) S_{{22}}(\mbf{w})+ \left( s_{{4}}+s_{{31}}+s_{{3}
}+s_{{21}}+s_{{11}} \right) S_{{211}}(\mbf{w}) + }\\[4pt]
&&\bleu{ \left( s_{{31}}+s_{{111}}
 \right) S_{{1111}}(\mbf{w}),}
\end{array}
\end{equation}
in which we miss only (see~\pref{multigraded})  the terms $\bleu{s_5\,S_{211} (\mbf{w}) + (s_{41}+s_6)\,S_{1111} (\mbf{w})}$. For sure, many more terms will be missing for larger $n$, since the maximal degree is $\binom{n}{2}$.
 
%%%%%%%%%%%%%%%%%%%%%%%%%%%%%%%%%%%%%%%%%%%%
%%%%%%%%%%%%%%%%%%%%%%%%%%%%%%%%%%%%%%%%%%%%
\section{Concluding remarks and thanks}
The proof of Theorems \ref{mainthmfrob} may readily be adapted to show that all isotypic components of the space $\Harm$ afford a universal Schur-positive expansion, for all finite complex reflection group $W$. 
 
We would like to thank Christian Stump for his independent verifications of some of our experiments leading to this paper. We would also like to thank Mark Haiman for enlightening discussions.

%%%%%%%%%%%%%%%%%%%%%%%%%%%%%%%%%%%%%%%%%%%%
%%%%%%%%%%%%%%%%%%%%%%%%%%%%%%%%%%%%%%%%%%%%


\begin{thebibliography}{10}   

\bibitem{livre} 
\auteur{F.~Bergeron},
\titreref{Algebraic Combinatorics and Coinvariant Spaces}, CMS Treatise in Mathematics, CMS and A.K.Peters,  2009.

\bibitem{trivariate} 
\auteur{F.~Bergeron and L.-F. Pr\'eville-Ratelle},
\titreref{Conjectures about Higher Trivariate Diagonal Harmonics via generalized Tamari Posets}, submitted.
(see arXiv:1105.3738v1)

\bibitem{cherednik}
\auteur{I.~Cherednik}, 
\titreref{Diagonal Coinvariants and Double Affine Hecke Algebras}, 
IMRN International Mathematics Research Notices 2004, No. 16., 769--791.

 \bibitem{chevalley}
\auteur{C.~Chevalley}, 
\titreref{Invariants of Finite Groups Generated by Reflections}, 
  Amer. J. Math. \vol{77} (1955), 778--782.
  
\bibitem{garsia_haiman} 
  \auteur{A.M.~Garsia and M.~Haiman},
 \titreref{A remarkable $q,t$-{C}atalan sequence and $q$-{L}agrange
  inversion}, J. Algebraic Combin. \vol{5} (1996), no.~3, 191--244.

\bibitem{gordon} 
\auteur{I.~ Gordon}, 
\titreref{On the quotient ring by diagonal invariants},
 Invent. Math., \vol{153} (2003), 503--518.
 
\bibitem{griffeth}
 \auteur{S.~Griffeth},
\titreref{Towards a combinatorial representation theory for the rational Cherednik algebra of type $G(r, p, n)$},
Proceedings of the Edinburgh Mathematical Society (Series 2) (2010), \vol{53},  419--445

\bibitem{HHLRU}
\auteur{J.~Haglund, M.~Haiman, N.~Loehr, J.~Remmel, and A.~Ulyanov},
\titreref{A Combinatorial Formula for the Character of the Diagonal Coinvariants},
Duke J. Math. 126 (2005), 195--232.
\bibitem{haimanhilb}    {\sc M. Haiman}, {\em  Combinatorics, symmetric functions and Hilbert schemes},  In CDM 2002: Current Developments in Mathematics in Honor of  Wilfried Schmid \& George Lusztig, International Press Books  (2003) 39--112.

\bibitem{haimanvanishing}
 {\sc M. Haiman}, {\em  Vanishing theorems and character formulas for the Hilbert scheme of points
in the plane}, Invent. Math. 149 (2002), 371--407.

\bibitem{macdonald} 
\auteur{I.~G. Macdonald},
\titreref{Symmetric functions and {H}all polynomials}, second ed., Oxford
  Mathematical Monographs, The Clarendon Press Oxford University Press, New
  York, 1995, With contributions by A. Zelevinsky, Oxford Science Publications.
  
\bibitem{loktev}
\auteur{S.~Loktev},
\titreref{Weight Multiplicity Polynomials for Multi-Variable Weyl Modules},
Mosc. Math. J., 2010, Volume 10, Number 1,  215--229.

\bibitem{sheppard} 
\auteur{G.C.~Shephard and J.A.~Todd}, 
\titreref{Finite Unitary Reflection Groups}, 
Canadian Journal of Mathematics \vol{6} (1954),  274--304.

  \bibitem{steinberg}
\auteur{R.~Steinberg},
\titreref{Invariants of Finite Reflection Groups},
     Canad. J. Math. \vol{12} (1960), 616--618.

\bibitem{weyl}
\auteur{H. Weyl},
\titreref{The Classical Groups, Their Invariants and Representations},
Second Edition, Princeton University Press, 1973.

\end{thebibliography}
\end{document}